\documentclass[10pt]{amsart}

\usepackage{amsmath, amsfonts, amssymb, amsthm, graphics}

\newtheorem{theorem}{Theorem}[section]
\newtheorem{lemma}[theorem]{Lemma}

\newtheorem{proposition}[theorem]{Proposition}

\theoremstyle{definition}

\newtheorem{remark}[theorem]{Remark}

\numberwithin{equation}{section}

\begin{document}

\title[A condition equivalent to the     H\"{o}lder   continuity]{A condition equivalent to the H\"{o}lder continuity of
harmonic functions on unbounded Lipschitz domains}

\author{Marijan  Markovi\'{c}}

\address{Faculty of Science and Mathematics\endgraf  University of Montenegro\endgraf  D\v{z}ord\v{z}a Va\v{s}ingtona BB
\endgraf 81000 Podgorica\endgraf Montenegro}

\email{marijanmmarkovic@gmail.com}

\begin{abstract}
Our main result concerns the behavior of bounded harmonic functions on a domain in   $\mathbb{R}^N$         which may be
represented as a strict epigraph of a Lipschitz function on $\mathbb{R}^{N-1}$. Generally speaking,      the result says
that the H\"{o}lder continuity of a harmonic function on such a domain is equivalent to the  uniform          H\"{o}lder
continuity along the straight lines determined by the vector  $\mathbf{e}_N$,                                      where
$\mathbf{e}_1,\mathbf{e}_2,\dots,\mathbf {e}_N$ is the base of standard vectors in $\mathbb{R}^N$.

More precisely, let $\Psi$ be a Lipschitz function on     $\mathbb {R}^{N-1}$, and $U$ be a real-valued bounded harmonic
function on $E_\Psi=\{(x',x_N): x'\in\mathbb{R}^{N-1}, x_N>\Psi(x')\}$.          We show that for   $\alpha\in(0,1)$ the
following  two conditions on   $U$   are equivalent:

(a) There exists a constant $C$ such that
\begin{equation*}
| U(x',x_N)  -   U(x',y_N)|\le C |x_N - y_N|^\alpha,\quad x'\in \mathbb {R}^{N-1}, x_N, y_N > \Psi (x');
\end{equation*}

(b) There exists   a constant $\tilde {C}$ such that
\begin{equation*}
|U(x) - U (y)|\le \tilde{C}  |x-y|^\alpha,\quad x, y\in E_\Psi.
\end{equation*}
Moreover,   the constant $\tilde {C}$ depends linearly on $C$.

The result holds as well for vector-valued harmonic functions and, therefore, for analytic mappings.
\end{abstract}



\maketitle

\section{Introduction and the main result}

In this paper $\alpha$ is a real number in $(0,1)$. We denote by $\Lambda^\alpha(E)$  the Lipschitz class of real-valued
functions on a non-empty set $E\subseteq \mathbb{R}^N$.                A function $f$ belongs to this    class  if it is
$\alpha$-H\"{o}lder  continuous, i.e.,  there exists  a  constant $C$ such that
\begin{equation*}
|f(x)  -  f (y)| \le  C  |x-y|^{\alpha},\quad x, y\in E.
\end{equation*}

In 1997, Dyakonov \cite{DYAKONOV.ACTM} proved that an analytic function on  the  unit disk $\mathbb {B} ^2$,  continuous
on the      closed unit disk,        is $\alpha$-H\"{o}lder continuous on $\mathbb {B} ^2$ if and only if its modulus is
$\alpha$-H\"{o}lder  continuous on the unit circle, i.e.,  $|f|\in \Lambda ^ \alpha (\partial \mathbb {B}^2)$,   and the
following  condition   is  satisfied
\begin{equation}\label{EQ.DYAKONOV.RADIAL}
| |f(r\zeta) | - |f(\zeta)|  |\le c (1-r)^ \alpha,\quad r\in (0,1),\zeta\in \partial \mathbb {B}^2,
\end{equation}
where  $c$ is a constant.  Pavlovi\'{c}  \cite{PAVLOVIC.ACTM}  gave a new proof of this result.

In 2007, Pavlovi\'{c} \cite{PAVLOVIC.RMI} establish  a counterpart of the    Dyakonov theorem   for real-valued harmonic
functions on  the  unit  ball  $\mathbb{B}^N\in \mathbb{R}^N$, continuous on the closure of the domain. He   showed that
it suffices to assume that the modulus of a harmonic function                                       belongs to the class
$\Lambda ^\alpha (\partial \mathbb{B}^N)$, or that the Lipschitz type  condition \eqref{EQ.DYAKONOV.RADIAL}           is
satisfied for $r\in (0,1)$ and $\zeta\in \partial \mathbb {B}^N$,                       to conclude that the function is
$\alpha$-H\"{o}lder continuous on $\mathbb{B}^N$.

It is natural  to ask  what can be said  for harmonic functions on  domains  in $\mathbb{R}^N$ different  from the  unit
ball.           It seems that one cannot give an easy answer even in the case of bounded domains with a smooth boundary.
Ravisankar considers this in his work \cite{RAVISANKAR.CVEE},      as well as in his Ph.D. thesis \cite{RAVISANKAR.PHD}.
Here, the author considers the behavior of a     function along transverse curves with respect to the boundary of    the
domain.

This work is devoted to the study of the Lipschitz-type condition posed on bounded harmonic functions defined on      an
unbounded domain in $\mathbb{R}^N$, which boundary is the graph of a Lipschitz function on $\mathbb{R}^{N-1}$.         A
domain of the such type is called  the strict epigraph of a function.       Our main theorem states that if a   harmonic
function is  $\alpha$-H\"{o}lder continuous only along  straight lines parallel  to the $x_N$-axis,           then it is
$\alpha$-H\"{o}lder  continuous   globally.

The main theorem of this paper is stated in the following.     Its proof is given in the next section.

\begin{theorem}[The  main  theorem]\label{TH.MAIN}
Let $\Psi:\mathbb{R}^{N-1}\to\mathbb{R}$ be $L$-Lipschitz, i.e., assume that
\begin{equation*}
|\Psi (x) - \Psi (y) |\le L |x-y|,\quad x,y\in \mathbb {R}^{N-1}.
\end{equation*}
If for  a bounded real-valued harmonic function $U$ on
\begin{equation*}
E_\Psi = \{x = (x',x_N)\in \mathbb{R}^N :x'\in \mathbb{R}^{N-1}, x_N> \Psi (x')\} \subseteq \mathbb{R}^N,
\end{equation*}
there exists a constant $C$ such that
\begin{equation*}
|U(x', x_N)  -  U (x', y_N)| \le   C  |x_N- y_N|^\alpha,\quad  x'\in \mathbb{R}^{N-1}, y_N, x_N>\Psi (x'),
\end{equation*}
where  $\alpha \in (0,1)$, then there  are   constants   $C_1$   and  $C_2$  such that
\begin{equation*}\begin{split}
\left|\frac  {\partial }{\partial x_N} U(x)\right| \le C_1 d(x,\partial E_\Psi)^{\alpha-1},
\quad |\nabla U (x)| \le   C _2  d(x,\partial E_\Psi) ^{\alpha-1} ,\quad x\in E_\Psi.
\end{split}\end{equation*}
Moreover,  $U\in \Lambda^\alpha (E_\Psi)$, i.e., there exists a constant $C_3$ for which we have
\begin{equation*}
|U(x) - U (y)| \le  C_3   |x-y|^{\alpha},\quad x,y\in E_\Psi.
\end{equation*}
In particular, $U$ has a continuous extension on the epigraph of $\Psi$.     The constants    $C_i$, $i=1,2,3$  linearly
depend on $C$.  
\end{theorem}

Similar statements may be found in \cite{AIKAWA.BLMS},              and in the papers mentioned in the references there.
In particular,        it is  known that if a  harmonic  function satisfies   the $\alpha$-H\"{o}lder condition   on  the
boundary of a  bounded Lipschitz domain in $\mathbb{R}^N$, then the function  is  $\alpha$-H\"{o}lder continuous  on the
domain,   provided  that $\alpha<\gamma$ where the  constant  $\gamma \le 1$ depends on the domain.  This result is well
known in the case of the unit ball or the upper half-space \cite{STEIN.BOOK}.

\section{Proof of the  main theorem}

We shall start this section with some preliminaries and auxiliary results.       For the theory of harmonic functions on
domains in $\mathbb{R}^N$  we refer to \cite{AXLER.BOOK}.

We need the gradient estimate for bounded harmonic functions on  a domain in $\mathbb{R}^N$  which includes the distance
function.    Here we also mention the estimate of the derivative in an arbitrary direction. The derivative of a function
$U$ in the direction $l\in\mathbb{R}^N$,   $|l|=1$,  is denoted  by
\begin{equation*}
\frac{\partial}{\partial \ell}U(x) = \lim_{t\to 0} \frac{U (x+t\ell ) - U(x)}{t}.
\end{equation*}
This   auxiliary result is a consequence of the Schwarz lemma for harmonic functions on the unit ball $\mathbb{B}^N$. It
will  be used several times in the proof of the main theorem. The lemma that follows after it concerns the estimate from
below of the distance of $x+\lambda\mathbf{e}_N$, $\lambda >0$, $x\in E_\Psi$, from the boundary of the         epigraph
$E_\Psi$ of the Lipschitz function $\Psi$. Here, we use some geometric considerations.

The content of the following proposition can be found in \cite[Theorem 6.26]{AXLER.BOOK}.

\begin{proposition}[Schwarz lemma for harmonic functions]\label{PROP.SCHWARZ.HARMONIC}
Let $U$ be a real-valued harmonic function on $\mathbb {B}^N$, $ |U(x)| \le 1$, $x\in \mathbb {B}^N$. Then
\begin{equation*}\label{EQ.GRAD.0}
|\nabla U (0)| \le  \frac {2 m_{N-1}(\mathbb {B}^{N-1})}{ m_ N(\mathbb {B}^N )},
\end{equation*}
where   $m_N$  is the Lebesgue measure on $\mathbb{R}^N$.
\end{proposition}

\begin{lemma}\label{LE.GRADIENT.DISTANCE}
Let $U$  be a bounded harmonic function with real value     on a domain $D\subseteq \mathbb{R}^N$. Then
\begin{equation*}
|\nabla U (x) |, \left|\frac {\partial}{\partial \ell}U(x)\right| 
\le  \frac{K_ N} {d(x,\partial D)} \sup_{y\in D} |U(y)|,\quad x\in D,\ell\in\mathbb{R}^N, |l|=1,
\end{equation*}
where
\begin{equation*}\begin{split}
K_ N & = \frac {2m_{N-1} (\mathbb{B}_{N-1})} {m_N (\mathbb{B} _{N})}.
\end{split}\end{equation*}
In particular,   this estimate holds for the $i$-th partial derivative             $\frac {\partial}{\partial x_i}U(x)$,
$1\le  i\le N$.

In case of an unbounded  domain  we have
\begin{equation*}
\lim _{d(x, \partial D) \to  \infty}  |\nabla U (x) | = 0.
\end{equation*}
\end{lemma}

\begin{proof}
Since $\frac{\partial}{\partial \ell}U(x)  =  \left<\nabla U(x),\ell\right>$, we have
\begin{equation*}
\left|\nabla U(x)\right| =    \sup_{ \ell\in \mathbb{R}^N, |\ell|=1}  \left| \frac{\partial}{\partial \ell}U(x) \right|,
\end{equation*}
so   it is enough to  prove   the gradient  estimate.

Assume that $U$ is not a  constant  function, and denote  $M  =  \sup_{y\in D}  |U(y)|$.              We shall  consider
\begin{equation*}
V (z) = M^{-1}  U (x  +  d(x,\partial D)  z),\quad   z\in \mathbb{B}^N.
\end{equation*}
It  is  a  real-valued  harmonic function  on  the  unit  ball. Since  $|V (z)|<1$,  $z\in \mathbb{B}^N$,  we may  apply
the gradient  estimate  given in  Proposition \ref{PROP.SCHWARZ.HARMONIC}.  By this theorem,  we have
\begin{equation*}
 M ^{-1}  d(x,\partial D)   |\nabla U(x)| = |\nabla V (0)|\le K_N.
\end{equation*}
From this inequality, we derive the gradient  estimate for  $U$  at  $x\in D$.
\end{proof}

In the sequel,       we denote by $B(a,r) = \{x\in\mathbb{R}^N:|x-a|<r\}$ the ball in $\mathbb {R}^N$ with the center at
$a\in \mathbb{R}^N$ and the radius $r>0$.

\begin{lemma}\label{LE.DISTANCE.FROM.GRAPH}
Let $\Psi: \mathbb{R}^{N-1}\to\mathbb{R}$ be a $L$-Lipschitz function.   For $x = (x',x_N)\in E_\Psi$ and $\lambda\ge 0$
let us denote
\begin{equation*}
x_\lambda  =  x  + \lambda \mathbf {e}_N= (x', x_N + \lambda)\in E_\Psi.
\end{equation*}
Then we have
\begin{equation*}
d(x_\lambda, \partial E_\Psi) \ge d(x, \partial E_\Psi) \cos \gamma +  {\lambda\cos \gamma},
\end{equation*}
where  $\gamma  = \arctan L$.
\end{lemma}

\begin{proof}
Denote  $d  =  d(x, \partial E_\Psi)$   and   $d'= d(x_\lambda, \partial E_\Psi)$.   There exists $P\in \partial E_\Psi$
such that $ d = |x - P|$  ($P$ is any point where the distance is achieved).

Assume first that $P = (x', \Psi (x'))$. Let $C_\gamma$ be the cone with the vertex at  $P$,   with the axis parallel to
$\mathbf{e}_N$ and the opening $\frac {\pi}2-\gamma$.                          Denote $R = d(x, \partial C_\gamma) $ and
$R'= d(x_\lambda,  \partial C_\gamma)$.     Since $\Psi$ is Lipschitz,    we have $d'\ge  R'$.   From the similarity  of
triangles,  we obtain the following proportion
\begin{equation*}
R': R = (d+\lambda):d,
\end{equation*}
from which  we  find    $R'  = \frac R d (d+\lambda)$. Since $\frac Rd = \sin (\frac \pi 2 - \gamma) = \cos \gamma$,  it
follows
\begin{equation*}
d'\ge R'\ge (d +\lambda)\cos \gamma,
\end{equation*}
which we aimed to prove.

Let us now  consider the case where $P$ is not of the preceding form, that is, $P =(y', \Psi (y'))$, $y'\ne x'$. Then we
can    obtain a better estimate from below of $d(x_\lambda, \partial E_\Psi)$.  We consider the intersection of $E_\Psi$
with the plane  $\Pi$, which is determined by the vectors $\Vec{xP}$ and $\mathbf{e}_N$.  The intersection of  $\Pi$ and
the boundary of the cone $C_\gamma $ is made up of two straight lines that form the angle      $\frac \pi2-\gamma$  with
$\mathbf{e}_N$.  In  this case one of them, let it be denoted by $l_\gamma$,                is the tangent to the circle
$\partial B(x,d)\cap  \Pi$  at  $P$. Denote   $R' = d(x_\lambda, l_\gamma)$.  From the  similarity of  triangles,   this
time   we  obtain
\begin{equation*}
R': d  = \left(\frac d{\cos \gamma} +\lambda\right) : \frac d{\cos \gamma}.
\end{equation*}
From this proportion we have
\begin{equation*}
R' =  d + \lambda \cos \gamma\ge d \cos \gamma  + \lambda \cos\gamma,
\end{equation*}
which implies   the inequality in this  lemma, since   $d'\ge R'$.
\end{proof}

\begin{proof}[Proof of the main theorem]
This proof is separated into three parts. In $i$-th part of the proof we establish       the existence  of the  constant
$C_i = C_i (C)$, $i=1,2,3$, where $C$ is the constant from the  assumed  inequality
\begin{equation}\label{EQ.VERTICAL}
|U(x', x_N) - U (x', y_N)| \le C |x_N- y_N|^\alpha,\quad  x'\in \mathbb{R}^{N-1}, y_N, x_N>\Psi (x').
\end{equation}
During the proof  we  shall see that each   $C_i$, $i=1,2,3$,   depends  linearly  on $C$.

\textit{I part}.
Here we show the existence of the constant  $C_1$        in  the estimate of the  $N$-th  partial  derivative   of  $U$.

For  $\lambda > 0$  we  shall  consider  the following  (harmonic)  function
\begin{equation*}
U _\lambda(x) =  U(x_\lambda )  =  U (x', x_N + \lambda),\quad x  = (x',x_N) \in E_\Psi.
\end{equation*}
In  other  words,  $U_\lambda$ is the   translation of $U$ for the vector $(-\lambda) \mathbf{e}_N$.

Attach to  $U_\lambda$  the  number
\begin{equation}\label{EQ.ALAMBDA}
A_\lambda  = \sup _{x \in E_\Psi} d (x, \partial E_\Psi )^{1 - \alpha}
\left| \frac {\partial} {\partial x_N} U_\lambda (x)\right|.
\end{equation}
We show below  that $A_\lambda$ is indeed finite, then we shall find an upper bound for $A_\lambda$, $\lambda>0$,  which
does not depend on $\lambda$. This is enough for the existence proof of the constant $C_1$.       Note that $A_0$ should
correspond to the function $U$.     However, it  seems  that  its finiteness  cannot be shown  in the way which follows.
This  is the reason for  introducing  $U_\lambda$.

By applying  Lemma \ref{LE.GRADIENT.DISTANCE} and  Lemma \ref{LE.DISTANCE.FROM.GRAPH},                         we obtain
\begin{equation*}\begin{split}
d (x, \partial E_\Psi )^{1 - \alpha}  \left| \frac {\partial} {\partial x_N} U_\lambda (x)\right|  
& = d (x, \partial E_\Psi )^{1 - \alpha} \left| \frac {\partial} {\partial x_N}  U  (x_\lambda)\right|
\\&  \le \frac { d (x, \partial E_\Psi )^{1 - \alpha} }{ d (x_\lambda, \partial  E_\Psi ) } 
K_N \sup _{y\in  E_\Psi}|U(y)|
\\&  \le \frac { d (x, \partial E_\Psi )^{1 - \alpha} }{ d (x, \partial E_\Psi) +\lambda }
\frac { K_N }{\cos \gamma}\sup _{y\in E_\Psi}|U(y)|.
\end{split}\end{equation*}
Considering  separately the cases  $d(x,\partial E)<1$   and   $d(x,\partial E)\ge 1$,         it is  not hard to derive
\begin{equation*}
d (x, \partial E_\Psi )^{1 - \alpha}  \left| \frac {\partial} {\partial x_N}  U_\lambda(x)\right|
\le  \max \{1, \lambda ^{-1}\} \frac {K_N}{\cos \gamma} \sup _{y\in E_\Psi}|U(y)|.
\end{equation*}
From the last inequality we conclude that $A_\lambda$ is the supremum of a bounded set,         which implies that it is
a finite number. However,  the  right side in the  above estimate is not bounded  in $\lambda$,   so we shall proceed to
obtain  one which is independent of $\lambda$.

For $x=(x',x_N)\in E_\Psi$,  for the sake of simplicity,  let us denote  $R = d(x,\partial E_\Psi)$.        Applying the
Taylor expansion  on the one variable function      $s\to U_\lambda(x',s)$    around $x_N$,                assuming that
$s\in (\Psi(x'),x_N)$,  we obtain that there exists  $t\in (s,x_N)$  such that
\begin{equation*}\begin{split}
U_\lambda (x', s) & 
=  U_\lambda  (x', x_N) + \frac {\partial } {\partial x_N}  U_\lambda (x',x_N) (s - x_N)  
\\&+ \frac {\partial^2} {\partial x_N^2}  U_\lambda (x',t)  \frac {(s-x_N)^2}2.
\end{split}\end{equation*}
In the above expansion let us take $s = x_N - \gamma R$, where the number $\gamma \in (0,1)$ will be chosen letter.   We
obtain
\begin{equation*}\begin{split}
U_\lambda (x', x_N  -  \gamma  R)  &
=  U_\lambda  (x', x_N) + \frac {\partial } {\partial x_N}  U_\lambda (x',x_N) ( - \gamma R )
 \\&+ \frac {\partial^2} {\partial x_N^2} U_\lambda (x',t)  \frac {(-\gamma R)^2}2.
\end{split}\end{equation*}
From the last equation we have
\begin{equation*}\begin{split}
\gamma \frac {\partial} {\partial x_N}  U_\lambda(x',x_N) R &
= U_\lambda  (x', x_N) - U_\lambda (x', x_N - \gamma R)
\\&+\gamma^2 \frac {\partial^2} {\partial x_N^2}  U_\lambda(x',t)   \frac {R^2}2.
\end{split}\end{equation*}
Taking the   absolute values on both sides above, we derive 
\begin{equation*}\begin{split}
\gamma \left| \frac {\partial } {\partial x_N} U_\lambda (x',x_N) \right|  R &
\le | U_\lambda  (x', x_N) -  U_\lambda (x', x_N - \gamma  R) |
\\& + \gamma^2 \left|\frac {\partial^2} {\partial x_N^2}  U_\lambda(x',t) \right|  \frac {R^2} 2.
\end{split}\end{equation*}
Applying now   \eqref{EQ.VERTICAL},  since $\gamma\in (0,1)$  and  $t\in (x_N - \gamma R,x_N)$,               we  obtain
\begin{equation}\label{EQ.FIRST.PART.1}\begin{split}
\gamma\left| \frac {\partial } {\partial x_N} U _\lambda (x) \right|  R  & \le C R ^{\alpha} + \gamma^2
\sup_{y\in {B}(x, \gamma R)} \left|\frac {\partial^2} {\partial x_N^2} U_\lambda (y) \right|\frac {R^2}2.
\end{split}\end{equation}

We shall use Lemma  \ref{LE.GRADIENT.DISTANCE}  for  (the harmonic function) $\frac {\partial} {\partial x_N} U_\lambda$
on the domain  $B(x, \frac {1+\gamma}2 R)\subseteq  E_\Psi$ in   order  to estimate  the  the second term on  the  right
side of \eqref{EQ.FIRST.PART.1}. Since  for $y\in {B}(x,\gamma R)$ we  have
$d(y,\partial B(x,\frac {1+\gamma}2 R))\ge \frac{1-\gamma}2 R$,  it follows
\begin{equation*}
\left|\frac {\partial^2} {\partial x_N^2} U_\lambda(y)\right|  \le
\frac {2K_N}{(1-\gamma) R} \sup_{z\in {B}(x,\frac{1+\gamma}2 R)}
\left |\frac {\partial} {\partial x_N}U_\lambda(z)\right|.
\end{equation*}

Applying the last estimate in  \eqref{EQ.FIRST.PART.1}, we arrive  at the following one
\begin{equation*}
\gamma\left| \frac {\partial } {\partial  x_N}  U_\lambda (x) \right| R \le
C R ^{\alpha} + \gamma^2 \frac {2K_N}{(1-\gamma) R} \sup_{z\in {B}(x,\frac{1+\gamma}2 R)} 
\left |\frac {\partial} {\partial x_N} U_\lambda(z)\right|\frac {R^2}2 .
\end{equation*}
After dividing it  by $\gamma R$, we obtain
\begin{equation}\label{EQ.FIRST.PART.2}
\left| \frac {\partial } {\partial  x_N} U_\lambda (x) \right| \le \frac{C} \gamma R ^{\alpha-1}
+   {\gamma} \frac {K_N}{{1-\gamma}} \sup_{z\in {B}(x,\frac{1+\gamma}2 R)}
\left |\frac {\partial} {\partial x_N}U_\lambda(z)\right|.
\end{equation}

Note that for $z\in {B}(x,\frac{1+\gamma}2 R)$ we have $d(z, \partial E_\Psi) \ge \frac {1-\gamma}2 R $. Having in  mind
the relation  \eqref{EQ.ALAMBDA}, since
\begin{equation*}
\left| \frac {\partial} {\partial x_N} U _\lambda (z)\right|
\le A_\lambda d(z,\partial E_\Psi)^{ \alpha-1}
\le A_\lambda \left( \frac {1-\gamma}2 R\right)^{ \alpha-1},
\end{equation*}
from \eqref{EQ.FIRST.PART.2} we obtain
\begin{equation*}
\left| \frac {\partial } {\partial x_N} U_\lambda (x)\right|
\le \frac{C}\gamma R^{\alpha-1}
+  \gamma  \frac { K_N}{1-\gamma} A_\lambda \left(\frac {1-\gamma}2 R \right)^{\alpha - 1},
\end{equation*}
from  which follows
\begin{equation}\label{EQ.LAST.ESTIMATE}
R^{1-\alpha}\left| \frac {\partial } {\partial x_N} U_\lambda(x)\right|
\le \frac{C}\gamma +  \gamma \frac{K_N}{1-\gamma}  A_\lambda \left(\frac {1-\gamma}2\right)^{ - 1}.
\end{equation}

Taking the  supremum on the left side  of   \eqref{EQ.LAST.ESTIMATE}   with  respect to  $x \in E_\Psi$ we have
\begin{equation*}\begin{split}
A_\lambda\le   \frac{C}\gamma + \frac {2\gamma K_N}{(1-\gamma)^2}  A_\lambda.
\end{split}\end{equation*}
At this moment we take  $\gamma \in (0,1)$ such that  $\frac {2\gamma K_N}{(1-\gamma)^2} = \frac 12$;      it is easy  to
check that  this quadratic  equation in $\gamma$  has  an  unique solution in  $(0,1)$.      From the last estimate   of
$A_\lambda$, since we are sure that   $A_\lambda$ is  finite, we obtain  $A_\lambda\le  \frac 2\gamma C$.       A simple
computation gives  the value  of  $\gamma$,  and  the estimate   $\frac 2\gamma\le  7  K_N$.        Therefore,  for  the
constant $C_1$ we set
\begin{equation*}
C_1 = 7 K_N C.
\end{equation*}

Since the estimate of  $A_\lambda$,  $\lambda>0$, we have just obtained, does not depend on $\lambda$, in the inequality
\begin{equation*}
\left| \frac {\partial} {\partial x_N} U (x',x_N+\lambda)\right|
= \left| \frac {\partial} {\partial x_N} U_\lambda(x',x_N) \right|
\le C_1 d(x,\partial E_\Psi)^{ \alpha-1},
\end{equation*}
we may let $\lambda\to 0$. In this way  we derive the following estimate  of the  $N$-th   partial  derivative
\begin{equation*}
\left| \frac {\partial} {\partial x_N} U(x',x_N) \right| \le C_1 d(x,\partial D)^{ \alpha-1},
\quad x=(x',x_N)\in E_\Psi,
\end{equation*}
which we aimed  in  this part of the proof.  As  one may expected,         the constant $C_1$   does not depend on  $L$.

\textit{II part}.  We   start now proving the  existence of the constant $C_2$     for  the gradient  estimate  of  $U$.

For $x=(x', x_N)\in E_\Psi$, as before, denote $R   =    d(x, \partial  E_\Psi)$.   Let us chose  arbitrary  a direction
vector $\ell\in \mathbb{R}^N$, $|\ell|=1$. Applying Lemma \ref{LE.GRADIENT.DISTANCE}  on         (the harmonic function)
$\frac{\partial}{\partial x_N} U$      on the ball  $B(x,\frac R 2)$ as a domain, then the just obtained estimate of the
$N$-th    partial derivative,    we  derive
\begin{equation*}\begin{split}
\left|\frac {\partial^2   } {\partial \ell \partial x_N}  U(x) \right|
&\le\frac {K_N}{\frac {R}2} \sup_{y \in {B}(x,\frac {R}2)}\left|\frac{\partial  }{\partial x_N} U(y)\right|
\\&\le \frac{2 K_N} {R}   C _1  \left(\frac{R}2\right)^{\alpha-1}
\\&= 2^{2-\alpha} K_N C_1 R^{\alpha-2},
\end{split}\end{equation*}
since for $y \in {B}(x,\frac {R}2)$ we have $d(y,\partial E_\Psi)\ge\frac {R}2$.

After we permute the $N$-th partial  derivative and the derivative in the direction $\ell$, we obtain
\begin{equation*}
\left|\frac {\partial^2 } {\partial x_N\partial \ell}  U (x)\right| 
\le  4 K_N  C_1  d(x, \partial  E_\Psi)^{\alpha-2},\quad x=(x',x_N)\in E_\Psi.
\end{equation*}
By applying  the  Fundamental Theorem of Calculus,  we  derive   (note  that  $x_0 = x$)
\begin{equation}\label{EQ.PART2}\begin{split}
\left|\frac {\partial U} {\partial {\ell }}(x_\lambda) -\frac {\partial U} {\partial {\ell }} (x_0)\right| &
=\left|\frac {\partial} {\partial {\ell }}
U(x',x_N + \lambda) -\frac {\partial} {\partial {\ell }} U (x',x_N)\right|  
\\&=\left|\int_{0}^{\lambda}  \frac {\partial^2 U} {\partial x_N\partial \ell} (x_t) dt \right|
\\&\le  4K_N  C_1 \int_{0}^{\lambda} d(x_t,\partial E_\Psi)^{\alpha-2}dt.
\end{split}\end{equation}
We  estimate  the last integral by applying Lemma \ref{LE.DISTANCE.FROM.GRAPH}.  Since by this lemma
\begin{equation*}
d(x_t,\partial E_\Psi)\ge (d(x,\partial E_\Psi)  + t)\cos \gamma ,\quad 0\le t\le \lambda,
\end{equation*}
we have
\begin{equation*}\begin{split}
\int_{0}^{\lambda} d(x_t,\partial E_\Psi)^{\alpha-2}dt &
\le (\cos \gamma) ^{\alpha-2} \int_ 0 ^\lambda ( d(x,\partial E_\Psi) + t)^{\alpha-2} dt
\\& = \frac {(\cos \gamma) ^{\alpha-2}} {1-\alpha}
\left(d(x,\partial E_\Psi)^{\alpha-1}  -( d(x,\partial E_\Psi) + \lambda)^{\alpha-1}\right).
\end{split}\end{equation*}
Let us  denote
\begin{equation*}
C_2 = \frac {4 K_N  C_1}  {(1-\alpha)(\cos\gamma)^{2-\alpha}}.
\end{equation*}
Now,  from  \eqref{EQ.PART2}, we derive
\begin{equation}\label{EQ.PART2.2}
\left|\frac {\partial }{\partial {\ell }}U(x  _ \lambda)  -\frac {\partial } {\partial {\ell }}  U (x)\right|
\le  C_2\left(d(x,\partial E_\Psi)^{\alpha-1}  -  ( d(x,\partial E_\Psi) + \lambda)^{\alpha-1}\right).
\end{equation}

Since, $\lim_{\lambda \to \infty } {d(x_\lambda,\partial E_\Psi)} = \infty$    (which follows from the distance estimate
from  below  given in  the Lemma \ref{LE.DISTANCE.FROM.GRAPH}), according         to the second  conclusion given in the
Lemma \ref{LE.GRADIENT.DISTANCE}, we have
\begin{equation*}
\lim_{\lambda\to \infty} |\nabla U(x_\lambda)| = 0.
\end{equation*}
The  same holds  for derivative  in direction  $\ell$ instead  of the gradient. Therefore, if we let $\lambda\to \infty$
in  \eqref{EQ.PART2.2}, we obtain
\begin{equation*}
\left|\frac{\partial }{\partial \ell} U (x)\right|\le C_2  d(x,\partial E_\Psi) ^{\alpha-1}.
\end{equation*}
Since this  estimate holds for every direction  $\ell$,  we derive the gradient estimate
\begin{equation*}\begin{split}
|\nabla U (x)|  & = \sup_{\ell\in \mathbb{R}^N, |\ell|=1}
\left|\frac{\partial }{\partial \ell} U(x)\right| \le C_2  d(x,\partial E_\Psi)  ^{\alpha-1},
\quad x = (x',x_N)\in E_\Psi.
\end{split}\end{equation*}

\textit{III part}.
We  shall proceed now to prove that $U$ belongs to the Lipschitz  class  $\Lambda_\alpha (E_{\Psi})$,  i.e., that  there
exists  the constant $C_3$.

Let $x = (x',x_N)\in E_\Psi$,  $y = (y',y_N)\in E_\Psi$ be arbitrary. Chose    $\lambda>0$ and  $\lambda'>0$   such that
$N$-th  coordinates of  $x_\lambda$ and $y_{\lambda'}$ are  equal, i.e.,   such that  $x_N+\lambda = y_N +\lambda'$, and
\begin{equation*}
\min\{\lambda ,\lambda' \}  =  (1+L) |x-y|.
\end{equation*}
Since  $\Psi$ is $L$-Lipschitz, the segment
$[x_\lambda, y_{\lambda'}]   =   \{t x_\lambda +    (1-t)y_{\lambda'}: t\in [0,1]\}$    belongs  to the domain $E_\Psi$.
For  the same reason,  we  have  $d(z,\partial E_\Psi)\ge  |x-y|$,  $z\in [x_\lambda, y_{\lambda'}]$,  so having in mind
the proved  gradient estimate,  it follows
\begin{equation*}\begin{split}
\int _ {[x_\lambda,y_{\lambda'}]} |\nabla  U( z )| |dz| & 
\le C_2 \int _ {[x_\lambda,y_{\lambda'}]} d(z,\partial E_\Psi)^{\alpha-1} |dz|
\\&\le C_2  | y_ {\lambda'} - x_{\lambda}|    |x-y|^{\alpha-1}
\\&\le  C_2 |x-y|^\alpha.
\end{split}\end{equation*}

Now, we have
\begin{equation*}\begin{split}
|U(x) - U (y)| & \le |U(x) - U (x_\lambda)| + |U(x_\lambda) - U (y_{\lambda'})| + |U(y_ {\lambda '}) - U (y)|
\\& \le C |x - x_\lambda |^\alpha  + \int _ {[x_\lambda,y_{\lambda'}]} |\nabla  U( z )| |dz|
+ C|y-y_{\lambda'}|^\alpha
\\&  \le  C (\lambda^\alpha  + {\lambda' }^\alpha ) +C_2 |x-y|^\alpha.
\end{split}\end{equation*}

Assume that $\lambda\le \lambda'$. Then $\lambda = (1+L) |x-y|$.     Since $\lambda '= \lambda  +  x_ N - y_N$, we  have
$\lambda ' \le \lambda +|x-y| $. It  follows
\begin{equation*}\begin{split}
|U(x) - U (y)| & \le   C (2\lambda^\alpha  + |x-y|^\alpha ) +C_2 |x-y|^\alpha
\\&=  (2 C (1+L)^\alpha  +  C + C_2 ) |x-y|^\alpha.
\end{split}\end{equation*}
Therefore, for the constant   $C_3$   we may take  $C_3 =  (2L+3)C + C_2 $.
\end{proof}

\begin{remark}
A version of the our main result probably holds  for harmonic functions on more general Lipschitz domains than   domains
which are strict epigraphs of Lipschitz  functions on    $\mathbb{R}^ {N-1}$.         This question will  be  considered
elsewhere. We give here some remarks on the proof of the main theorem which concern possible extensions of the   result.

It seems that  the existence  of the constant  $C_1$ may be proved in a very  general setting, even when we do not  have
any strict  conditions  on the  boundary of a  domain.

As  may be readly  seen, the main problem in the course of a   potential extension of the result is  the  existence   of
the constant $C_2$ for the gradient estimate. In our  proof  the crucial role plays the behavior of the gradient  of   a
bounded harmonic function on an unbounded domain.     At this moment it is not clear  how to overcome this difficulty in
case of  bounded domains.

If  one  establishes the existence of the constant  $C_2$, the existence of  $C_3$ is straightforward in case         of
Lip$_\alpha$-extension domains.  For the exposition on this type of domain    we refer  to Lappalainen      dissertation
\cite{LAPPALAINEN.AASF}.  Below  we briefly  mention  some facts.

On  the  Lipschitz  class  $\Lambda^\alpha (E)$  the norm  may be   introduced by
\begin{equation*}
\|f\|_{\alpha} = \sup _{x,  y\in E,\, x\ne y}  \frac {|f(x) - f(y)|}{|x-y|^\alpha}.
\end{equation*}
On the other hand, the local Lipschitz class $\Lambda^\alpha _ \mathrm {loc}(E)$ contains functions $f$ on $E$ for which
\begin{equation*}
\|f\|_{ \mathrm{loc}, \alpha} = \sup _{x,  y\in E,\,  x\ne y,\,  d(x,y)< \frac {d(x,\partial E)}2} 
\frac {|f(x) - f(y)|}{|x-y|^\alpha}.
\end{equation*}
is  finite. This number is defined  to be  the  norm in   $\Lambda^\alpha _ \mathrm {loc}(E)$.

By Gehring and Martio \cite{GEHRING.AASF},                             a domain $D\subseteq \mathbb {R}^N$ is called the
$\mathrm {Lip}_\alpha$-extension  domain if    $\Lambda^\alpha_\mathrm{loc}(D) =\Lambda^\alpha(D)$ along  with  the norm
equivalence.  They proved the  following  criterion: A domain $D$ is the $\mathrm {Lip}_\alpha$-extension domain if  and
only if there exists a constant $C_D$ such that for every     $x, y\in D$ may be connected by       a rectifiable  curve
$\gamma\subset D$ such that
\begin{equation}\label{EQ.CD}
\int_\gamma d(z,\partial  D) ^{\alpha-1} |dz|\le C_D |x-y|^\alpha.
\end{equation}

For example,   uniform, or John   domains  in  $\mathbb{R}^N$  are Lip$_\alpha$-extension domains. It is known that  the
class  of   uniform domains  contains  the class of Lipshitz domains. Since the function  $\Psi$    in the main  theorem
is    Lipschitz,   $E_\Psi$ is the  Lip$_\alpha$-extension  domain. By modifying of  the last part of the  proof of  the
main  theorem one shows that $E_\Psi$  is the  Lip$_\alpha$-extension  domain. Therefore, the existence of the  constant
$C_3$ may be alternatively derived from the following result proved       by the author quite recently in a very general
setting -- for regularly  oscillating  mappings between  metric spaces \cite{MARKOVIC.JGA}.

\begin{proposition}
Let $D\subseteq \mathbb {R}^N$ be a Lip$_\alpha$-extension domain.  If a function  $f$ is differentiable on $D$, and  if
it  satisfies
\begin{equation*}
|\nabla f(z)| \le C_f d(z, \partial D)^{\alpha-1},\quad z\in  D,
\end{equation*}
where $C_f$ is a constant, then   $f$   belongs to the class $\Lambda^\alpha(D)$,  and we have
\begin{equation*}
|f(x)- f(y)|\le  C_D C_f |x-y|^{\alpha},\quad x, y\in D,
\end{equation*}
where $C_D$ is the constant  in \eqref{EQ.CD}.
\end{proposition}
\end{remark}

\subsection*{Acknowledgment}
The       author thanks the reviewer for the detailed reading of this work and for the comments that contributed to  the
higher quality of the  paper.


\begin{thebibliography}{10}

\bibitem{AIKAWA.BLMS}
H. Aikawa, \textit{H\"{o}lder continuity of the Dirichlet solution for a general     domain},     Bulletin of the London
Mathematical Society \textbf{34} (2002) 691--702.

\bibitem{AXLER.BOOK}
Sh. Axler, P. Bourdon, W. Ramey,     \textit{Harmonic function theory}, second edition, Springer-Verlag, New York, 2001.

\bibitem{DYAKONOV.ACTM}
K. Dyakonov,            \textit{Equivalent norms on Lipschitz-type spaces of holomorphic functions},    Acta Mathematica
\textbf{178} (1997), 143--167.

\bibitem{DYAKONOV.AIM}
K. Dyakonov, \textit{Holomorphic functions and quasiconformal mappings with smooth moduli},     Advances in  Mathematics
\textbf{187} (2004), 146--172.

\bibitem{GEHRING.AASF}
F. Gehring,  O. Martio, \textit{Lipschitz classes and quasiconformla mappings},  Annales        Academi{\ae} Scientiarum
Fennic{\ae} \textbf{10} (1985), 203--219.

\bibitem{LAPPALAINEN.AASF}
V. Lappalainen, \textit{Lip$_h$-extension domains},  Annales  Academi{\ae} Scientiarum Fennic{\ae} Dissertationes, 1985.

\bibitem{MARKOVIC.JGA}
M. Markovi\'{c}, \textit{Regularly oscillating mappings  between  metric  spaces and a theorem of Hardy and Littlewood},
The Journal of Geometric Analysis \textbf{34} (2024), article number 165.

\bibitem{PAVLOVIC.ACTM}
M. Pavlovi\'{c}, \textit{On Dyakonov's paper ''Equivalent norms on  Lipschitz-type  spaces of  holomorphic functions''},
Acta Mathematica  \textbf{183} (1999), 141--143.

\bibitem{PAVLOVIC.RMI}
M. Pavlovi\'{c},  \textit{Lipschitz conditions on the modulus of a harmonic function}, Rev.      Mat.     Iberoamericana
\textbf{23}  (2007), 831--845.

\bibitem{PAVLOVIC.BOOK}
M. Pavlovi\'{c}, \textit{Function classes on the unit disc}, De Gruyter Studies in Mathematics, 2019.

\bibitem{RAVISANKAR.PHD}
S. Ravisankar, \textit{Lipschitz properties of harmonic and holomorphic functions}, Ph.D. dissertation,  The Ohio  State
University,  2011.

\bibitem{RAVISANKAR.CVEE}
S. Ravisankar, \textit{Transversally     Lipschitz harmonic functions are Lipschitz},    Complex Variables and  Elliptic
Equations \textbf{58} (2013), 1685--1700.

\bibitem{STEIN.BOOK}
E.  Stein, \textit{Singular Integrals and Differentiability Properties of Functions }, Princeton University Press, 1970.

\end{thebibliography}
\end{document}